\documentclass[a4paper,12pt]{amsart}
\usepackage{amsmath,amsthm,amssymb}
\usepackage{hyperref}

\allowdisplaybreaks[1]

\numberwithin{equation}{section}
\newtheorem{theorem}{Theorem}[section]

\newtheorem{proposition}[theorem]{Proposition}

\theoremstyle{remark}
\newtheorem{remark}[theorem]{Remark}
\newtheorem{example}[theorem]{Example}
\newtheorem{definition}[theorem]{Definition}

\newcommand{\w}{\omega}

\newcommand{\ity}{\infty}

\newcommand{\R}{\mathbb{R}}

\newcommand{\ti}{\widetilde}

\newcommand{\ta}{\theta}

\newcommand{\N}{\mathbb{N}}
\newcommand{\Z}{\mathbb{Z}}

\newcommand{\al}{\alpha}

\begin{document}

\title[On the Escaping Set in Topological Dynamics]
{On the Escaping Set in Topological Dynamics}

\author[K. Lalwani]{Kushal Lalwani}
\address{Kushal Lalwani\\Department of Mathematics\\ University of Delhi\\Delhi--110 007, India}
\email{lalwani.kushal@gmail.com }



\begin{abstract}
We wish to investigate some elementary problems concerning topological dynamics revolving around our proposed definition of escaping set. We also discuss the notion of escaping set in the induced dynamics of the hyperspace. Moreover, we extend the notion of limit set and escaping set for the general semigroup generated by continuous self maps.
\end{abstract}

\maketitle

\section{Introduction}
Topological dynamics is the study of  properties of dynamical systems those are invariant under topological conjugacy. In general the topological dynamical systems is the  study of transformation groups with respect to those topological properties whose prototype occurred in the abstract theory of classical dynamics. In the seminal work of Henri Poincar$\acute{e}$ in \cite{HP}  and G. D. Birkhoff in \cite{Bir}, the general setting is  differential equations in $\R^n$. In this  setting  those notions are studied  which  arise in the qualitative study of differential equations with compactness being a crucial hypothesis. The abstract theory of modern dynamical system is expounded by Gottschalk and Hedlund in \cite{GH}, Robert Ellis in \cite{ellis}, H. Furstenberg in \cite{HF}, and J. Auslander in \cite{JA}, etc. However,  these studies revolve around  the eventual behavior of the iterates of some self maps of a nice space such as a compact space. As one can observe that non-vacuous  results in this direction depends on the convergence of the sequence of iterates which is easily ensured by the compactness of the space. For noncompact topological space story is quite different. We see that very little is known in the case of noncompact spaces. In this direction J. Auslander \cite{ja} generalized the work on Lyapunov stability of compact invariant sets and discussed the dynamical systems which have no generalized recurrent points.  Also M. Hurley in \cite{mh1, mh2} extended chain recurrence to  noncompact spaces.

Our purpose in this paper is to initiate a systematic investigation of the eventual behavior of the sequence of iterates of a point when the space is not compact. We will try to answer some natural questions such as : What will happen to the iterates? Does it converge? Does it have some limit points, that is, some convergent subsequence? Or in the extreme case, does it move far away to escape from the space?  Speaking metaphorically,  it goes towards the so called infinity. The situation in the last case is possible only if the space is not compact. If such a thing happens we call that point to be an {\it{escaping point}} of the dynamical system. As far as we know, the concept of escaping set was first introduced by A. E. Ermenko \cite{er} in the field of complex dynamics. Also John M. Lee \cite{lee} discussed it for some other purpose and in a different sense for sequences in the context of proper maps. In fact, these works propelled us towards this study. We investigate the concept of escaping set in general setting for a topological space.  We provide some examples of escaping set to motivate the study. 

We also discuss the relationship of escaping set in a dynamical system $(X,\Phi)$ to its counterpart in the induced dynamics on $(K(X),\Phi_K)$, the hyperspace of all compact subsets of $X$. It is well known that compactness and noncompactness are genetic properties of  $X$ and $K(X)$, that is, $K(X)$ is compact if and only if $X$ is compact. It was the reason we compare the escaping set in the two dynamical systems. Naturally escaping of point in dynamical systems has been replaced by the study of escaping of compact subsets in the induced system.  Interestingly, it turns out that a compact subset of $X$ escapes in the dynamics of $(K(X),\Phi_K)$  if and only if each element of the compact subset escapes in the dynamics of  $(X,\Phi)$.

Further we extend the notion of escaping set and the notion of convergence of iterates (the {\it{omega limit points}}) to the more general setting associated to an arbitrary semigroup of continuous self maps of a phase space. Hinkkanen and Martin \cite{hm} pioneered the notion of semigroup to extend the theory of complex dynamics, the notion was already there in some special setting of  a  flow. One sees in the discrete dynamical system where the semigroup was generated by a single function, in continuous dynamical system (or a flow) $\Phi :X \times \R_+ \to X$ is an abelian semigroup with infinitely many generators $<\Phi(.,t)>_{t\in(0,1]}$. Our aim here is to see how far the notion of escaping set and the classical theory of limit points applies in this general setting.

For general reference to standard terms and basic facts we follow  Alongi and Nelson\rq{}s book \cite{Alongi}. We would also like to mention that  one can refer E. Michael\rq{}s paper \cite{em} and the textbook by S. B. Nadler \cite{nad}  for  the theory of hyperspace.

\section{Definitions and Examples} \label{sec2}

For the sake of completeness we mention some standard definitions and results, although they can be found in the reference we have already mentioned in the introduction. This section also contains our proposed definition of escaping set with  examples.

A topological dynamical system or a flow is a triple $(X,T,\Phi)$ consisting of a topological space $X$ (the phase space) with a group (or a semigroup) $T$ (the phase group) and a mapping $\Phi :X \times T \to X$ such that\\
(i) $\Phi^t(.)=\Phi(.,t): X \to X$ is continuous for each $t\in T$, and\\
(ii) $\Phi(x,ts)=\Phi(\Phi(x,t),s)$ for all $t,s\in T$.

For our purpose, $T$ will be either the additive group of reals $(\R,+)$ or $(\Z,+)$ or $(\N_0,+)$, where $\N_0=\N \cup \{0\}$, the set of non negative integers. Then $\Phi^0$ is the identity map on $X$. Also if $T=\Z$ or $\R$, the map $\Phi^t$ is invertible (hence a homeomorphism) for each $t$ and $(\Phi^t)^{-1}=\Phi^{-t}$. The space $X$ is assumed to be Hausdorff and first countable.

\begin{definition}
Let $X$ and $Y$ be topological spaces. Two flows $(X,T,\Phi)$ and $(Y,T,\Psi)$ are topologically conjugate if there exists a homeomorphism $h:X\to Y$ such that $\Psi^t \circ h=h \circ \Phi^t$ for each $t\in T$.
\end{definition}

\begin{definition}
Let  $(X,T,\Phi)$ be a topological dynamical system. A point $z\in X$ is called an $\w$-limit point for $x \in X$ if there exists a sequence $\{t_n\}$ in $T$ such that $t_n \to \ity$ as $n \to \ity$ and $\Phi^{t_n}(x) \to z$.
\end{definition}

\begin{definition}
Let  $(X,T,\Phi)$ be a topological dynamical system. A point $z\in X$ is called an $\al$-limit point for $x \in X$ (provided $\Phi^t$ is invertible for every $t \in T$)  if there exists a sequence $\{t_n\}$ in $T$ such that $t_n \to \ity$ as $n \to \ity$ and $\Phi^{-t_n}(x) \to z$.
\end{definition}
The $\w$-limit set $\w(x)$, is the set of all  $\w$-limit points for $x$ . The $\al$-limit set $\al(x)$, is the set of all  $\al$-limit points for $x$ .

\begin{definition}\label{def}({\it Escaping Set})
Let  $(X,T,\Phi)$ be a topological dynamical system. A point $x\in X$ is called an escaping point if either\\
(i) for any compact $K\subset X$ there exists $T_K >0$ such that $\Phi^t(x)\notin K$ for each $t>T_K$; or\\
(ii) if $\Phi^t$ is invertible for every $t \in T$; for any compact $K\subset X$ there exists $T_K >0$ such that $\Phi^t(x)\notin K$ for each $t<-T_K$.\\
The set of all escaping points of $\Phi$ is denoted by $Esc(\Phi)$.
\end{definition}

We denote 
$$Esc^+(\Phi) :=\{x\in X : x \text{ satisfies only condition (i) above}\}$$ and $$Esc^-(\Phi) :=\{x\in X : x \text{ satisfies only condition (ii) above}\}.$$ Then $Esc(\Phi)=Esc^+(\Phi) \cup Esc^-(\Phi)$. Clearly, if $X$ is compact then $Esc(\Phi)$ is empty.

\begin{definition}
Let $(X,T,\Phi)$ be a flow. A subset $M$ of $X$ is a minimal set if $M$ is closed  nonempty and invariant, and if $M$ has no proper subsets with these properties.
\end{definition}

Note that a nonempty subset $M$ of $X$ is minimal if and only if it is the orbit closure of each of its points. For, if $M$ is minimal and $x \in M$, its orbit closure is closed invariant and nonempty. 

Now we shall provide some example of noncompact dynamical systems having escaping points. We begin with an elementary flow on $\R^n$ and then proceed to other interesting examples.

\begin{example}
Consider a constant differential equation on $\R^n$
$$\frac{dx}{dt}=c$$
for $c\in \R^n$. The flow corresponding to this differential equation is $(\R^n,\R,\Phi)$, where $\Phi:\R^n \times \R \to \R^n$ is defined as $\Phi(x,t)=x+ct$. Here each point of $\R^n$ is an escaping point.
\end{example}

\begin{example}
Consider the differential equations on $\R^2$
$$\frac{dx_1}{dt}=x_2+x_1(1-r^2),\quad$$
$$\frac{dx_2}{dt}=-x_1+x_2(1-r^2).$$
In the polar coordinates we have
$$\frac{dr}{dt}=r(1-r^2),$$
$$\frac{d\ta}{dt}=-1 \qquad \quad$$

The solution to above differential equations is given by
$$r(t)=\frac{e^tr(0)}{\sqrt{r^2(0)(e^{2t}-1)+1}}$$
and
$$\ta(t)=-t+\ta(0).\qquad \qquad$$

Here $Esc(\Phi)=\{x\in\R^2: |x|>1\}$.
\end{example}

\begin{example} \cite{Alongi}
Consider the following system of differential equations on  $\R^3$
$$\frac{dx}{dt}=\frac{2xz}{x^2+y^2+z^2+1}-y$$
$$\frac{dy}{dt}=\frac{2yz}{x^2+y^2+z^2+1}+x$$
$$\frac{dz}{dt}=\frac{z^2-x^2-y^2+1}{x^2+y^2+z^2+1}\qquad$$

Let $\Phi^t$ be the flow corresponding to the system.

Here $Esc(\Phi)=\{(x,y,z)\in \R^3:x=y=0\}$.
\end{example}

Next we construct	 an example of dynamics on space of infinite symbols.

\begin{example}
Consider the left shift on $\N^\N$ with product topology,
$$\Phi:\N^\N \to \N^\N $$
$$\Phi(x_1,x_2,x_3,\ldots)=(x_2,x_3,\ldots).$$

For $x=(x_1,x_2,x_3,\ldots)\in\N^\N$, consider

$\underline{Case\ 1:}$ $(x_n)$ is bounded in $\N$.\\
If $x_n\le K,\ \forall \ n$ then the orbit $O(x)\subset \{1,2,\ldots,K\}^\N$.

$\underline{Case\ 2:}$ $(x_n)$ is unbounded in $\N$.\\
Since  $(x_n)$ is unbounded there is a subsequence $(x_{n_j})$ diverging to $\ity$. Suppose the sequence of subscripts $(n_j)$ has bounded gaps, bounded by some $K \in \N$. For any compact $C \subset \N^\N $, take $m_i=max\{a:a\in \pi_i(C)\}$. Each $\pi_i$ is the projection of $\N^\N$ onto its $i$th factor. Since  $(x_{n_j})$ diverges to $\ity$, for $M=max\{m_i:1\le i \le  K\}$ there exists $N\in \N$  such  that  $ x_{n_j}> M\ \forall\ n_j \ge N$ . Thus $\Phi^m(x) \notin C$ whenever $m \ge N-K$ and non negative. In this case $(x_n)$ is an escaping point.

Otherwise, for $x=(1,2,1,4,1,1,1,8,1,1,1,1,1,1,1,16,1,1,\ldots)$, the subsequence $(\Phi^{2^j}x)$,  $j \in \N$  converge to $(1,1,1,1,\ldots)$. Thus $x$ does not escape.  Also, for $y=(1,2,1,4,3,2,1,8,7, $ $6,5,4,3,2,1,16,15,14,\ldots)$ and any compact $C \subset \N^\N $, take $m_i=max\{a:a\in \pi_i(C)\}$. Consider $M=max\{m_i:1\le i \le  m_1+1\}$ there exists $N\in \N$  such  that  $ 2^k> M\ \forall\ k \ge N$. Thus $\Phi^m(y) \notin C$ whenever $m \ge 2^N-m_1-1$ and non negative. Hence $y$ is an escaping point.

\end{example}

\section{Results}\label{sec3}

In this section we will discuss some elementary properties of escaping set which will be useful. Finally we consider the escaping set on the corresponding hyperspace. We proceed by showing that the property of a point to escape is topological invariant. 

\begin{theorem}
Let $X$ and $Y$ be topological spaces and $(X,T,\Phi)$ and $(Y,T,\Psi)$ be flows. If $h:X\to Y$ is a topological conjugacy from $\Phi^t$ to $\Psi^t$, then
$$Esc(\Psi)=h(Esc(\Phi)).$$
\end{theorem}

\begin{proof}
Let $x\in Esc(\Phi)$. Let $K$ be a compact subset of $Y$ such that $\Psi^t(h(x))\in K$, for each $t \in T$. Since $h$ is a topological conjugacy, we have, $h(\Phi^t(x))\in K$, that is, $\Phi^t(x)\in h^{-1}(K)$ for each $t\in T$, which is a contradiction. Therefore $h(Esc(\Phi))\subset Esc(\Psi)$

Also, for $y \in Esc(\Psi)$, let $y=h(x\rq{})$. If $K\rq{}$ is a compact subset of $X$ such that $\Phi^t(x\rq{})\in K\rq{}$, for each $t \in T$, then this gives $h(\Phi^t(x\rq{}))\in h(K\rq{})$ or $\Psi^t(h(x\rq{}))\in h(K\rq{})$, that is, $\Psi^t(y)\in h(K\rq{})$, for each $t \in T$. Thus $Esc(\Psi)\subset h(Esc(\Phi)).$ Therefore, $Esc(\Psi)=h(Esc(\Phi)).$
\end{proof}

\begin{proposition} \label{eq}
The set $\w (x)$ is empty if and only if $x \in Esc^+(\Phi)$.
\end{proposition}

\begin{proof}
Clearly,  $z$ is an escaping point if and only if for $n>0, \ (\Phi^n(x))_n$ does not contain any convergent subsequence, that is, $\w(x)=\emptyset$.
\end{proof}

On the similar lines one can deduce the corresponding result for $\al$-limit set.

\begin{proposition}
The set $\al (x)$ is empty if and only if $x \in Esc^-(\Phi)$.
\end{proposition}

\begin{remark}
We know that \cite{Alongi} every flow on a nonempty compact topological space has a minimal set. Also on a noncompact space, if $x\in Esc^+(\Phi) \cap Esc^-(\Phi)$ then orbit of $x$ is a minimal set for the flow.
\end{remark}

Next we consider the problem of escaping of a compact subset of a given dynamical system.

\begin{definition}
For a topological space ($X,\mathcal{T}$), we define the {\it hyperspace} $K(X):=\{A\subset X : A{\rm{\ is \ nonempty}}\ {\rm{and\ compact}} \}$. For any number of nonempty open sets $V_1,\ldots, V_n \in \mathcal{T}$, the family $\{<V_1\ldots ,V_n>\}$ forms a basis for the Vietoris topology on $K(X)$, where
\begin{equation}\notag
<V_1,\ldots,V_n>:=\{A\in K(X): A \subset \bigcup_{i=1}^n V_i\  {\rm{ and }}\  A\cap V_i\ne \emptyset \ \forall\  i=1,\ldots,n\}.
\end{equation}

\end{definition}

The continuous map $\Phi : X \to X$ induces a continuous map $\Phi_K : K(X) \to K(X)$ defined as 
\begin{equation}\notag
\Phi_K(A)=\Phi(A).
\end{equation}

\begin{theorem}[\cite{em}]
The space $X$ is compact (locally compact) Hausdorff  if and only if $K(X)$ is compact (locally compact) Hausdorff.
\end{theorem}

\begin{theorem}[\cite{em}]
$X$ is  first (second) countable if and only if $K(X)$ is  first (second) countable.
\end{theorem}

\begin{definition}
Let $(A_i)$ be a sequence of subsets of $X$. The $ \limsup A_i :=\{x\in X$ : for each $U$ open in $X$ such that $x\in U,\  U\cap A_i \ne \emptyset$ for infinitely many $ i\}$ .
\end{definition}
For any compact subset $A$ of $X$, we define
$$\w_K (A):= \limsup\ \Phi_K^i(A).$$

\begin{theorem}
Let $X$ be a locally compact Hausdorff space and $A$ is a compact subset of $X$. For a continuous map $\Phi:X\to X$ the following are equivalent:\\
1) For each $a \in A,\ a\in Esc(\Phi) ; $\\
2) For each $a \in A,\ \w (a)=\emptyset ; $\\
3) $\w_K(A)=\emptyset ;$\\
4) $A \in  Esc(\Phi_K).$
\end{theorem}

\begin{proof}
The equivalence of  (1) and (2); and (3) and (4) follows from proposition \ref{eq}. First we shall show (1) implies (3).

Since $X$ is locally compact Hausdorff space, we have for $x \in X$  there is a neighborhood $V$  of $x$ such that $\overline{V}$ is compact. For each $a \in A,\ \exists \  n_a \in \N$ such that $\Phi^i(a) \notin \overline{V}, \ \forall i \ge n_a$.\\
\underline{Claim:} The set $\{n_a\}$ is bounded above.\\
Otherwise for each $n_0>0\ \exists \ a_0 \in A$ such that $\Phi^m(a_0) \in \overline{V}$ for some $m>n_0$. That is, there is a sequence of points $(a_1,a_2,a_3,\ldots)$ in $A$ corresponding to $n_{1}>n_{2}>n_{3}>\ldots$. Since $A$ is compact the sequence $(a_i)$ has a convergent subsequence $(a_{i_j})$ converging to some $a' \in A$. Since for each $m, \ \Phi^m$ is continuous, we have $(\Phi^m(a_{i_j}))$ converges to $\Phi^m(a')$. For $\Phi^m(a')\notin \overline{V}, \ \forall\  m\ge n_{a'}$, there is a $N\in \N$ sufficiently large such that for each $i_j>N$ and $m \ge n_{a'}$ $\Phi^m(a_{i_j}) \notin \overline{V}$. Which is a contradiction. Thus the set $\{n_a\}$ is bounded above. 

Since  $\{n_a\}$ is bounded above there is a $N'$ sufficiently large such that
$$\qquad \Phi^i(a)\notin \overline{V},\ \forall \ a \in A, \ i\ge N',$$
$$or, \ \Phi^i(A) \cap V=\emptyset , \ \forall\  i \ge N',$$
$$or, \ \limsup  \Phi^i(A)=\emptyset .\qquad \quad$$

Clearly (3) implies (2), for if $\w(a)\ne \emptyset$ for some $a\in A$ then $\limsup \Phi^i(A)\ne\emptyset$.
\end{proof}

\section{Semigroup of continuous maps}

A {\it{continuous semigroup}}  is a set of (non-identity) continuous self maps of a topological space which is closed under the composition. A semigroup $G$ is said to be generated by a family $\{g_{\al}\}_{\al}$ of  continuous self maps of a topological space $X$  if every element of $G$ can be expressed as compositions of iterations of the elements of  $\{g_{\al}\}_{\al}$. We denote this by $G=<g_{\al}>_{\al}$.

\begin{definition}
Given a semigroup $G$ and $x \in X$, the set $O_G(x):=\{g(x): g \in G\}$ is called the (forward) {\it orbit} of $x$ under $G$.
\end{definition}

\begin{definition}
A sequence of functions $(f_{n_k})\subset G$ is said to be {\it{unbounded}} if  $n_k\to \ity$ as $k\to \ity$ and each $f_{n_k}$ consists of exactly $n_k$ iterates of  $g_{\al_{0}}$, for fix $g_{\al_0}\in \{g_{\al}\}_{\al}$, that is, $f_{n_k}=h_1\circ g_{\al_0}\circ h_2 \circ g_{\al_0} \circ h_3 \circ \ldots \circ h_{n_k}\circ g_{\al_0} \circ h_{n_k+1}$, where each $h_i \in G\cup \{identity\}$ and $h_i$s are independent of $g_{\al_0}$.
\end{definition}

Note that the unboundedness of a sequence in $G$ is not with respect to some metric on $G$. The term unbounded refers to the unboundedness of the sequence $(n_k)$, the number of iterates of a generator of $G$. Recall that in case of a discrete or continuous dynamical system any unbounded sequence corresponds to the iterates of $\Phi^1$ (that is, $\Phi^t$, when $t=1$, as mentioned in section \ref{sec2}).

\begin{definition}
A point $z\in X$ is called an {\it{$\w-$limit\ point}} for a point $x\in X$ if for some unbounded sequence $(f_{n_k})\subset G, \ f_{n_k}(x)\to z$ as $k \to \ity$.
\end{definition}

\begin{definition}
Let $G=\ <g_i>_{i \in \varLambda}$ be a semigroup of continuous self maps of $X$ and $\ti{G}=\ <\ti{g_i}>_{i \in \varLambda}$ be a semigroup of continuous self maps of $Y$. Two dynamical systems $(X,G)$ and $(Y,\ti{G})$ are said to be topologically conjugate if there exists a homeomorphism $\rho : X \to Y$ such that $\rho\circ g_i=\ti{g_i}\circ \rho$ for each $i \in \varLambda$.
\end{definition}

\begin{proposition} \label{conjugacy}
Let $(X,G)$ and $(Y,\ti{G})$ be two dynamical systems. If $\rho : X \to Y$ is a topological conjugacy then $\w(\rho(x))=\rho(\w(x))$ for every $x \in X$.
\end{proposition}

\begin{proof}
If $x \in X$ and $z\in \w(x)$ then there exists an  unbounded sequence $(f_{n_k})\subset G$ such that  $\displaystyle{\lim_{k\to \ity}}f_{n_k}(x)=z $. Since $\rho$ is a conjugacy
\begin{equation}\begin{split}\notag
\rho(z)&=\rho(\lim f_{n_k}(x)) \\
&=\lim \rho(f_{n_k}(x))\\
&=\lim\rho((h_1\circ g_{\al_0}\circ h_2\circ \ldots \circ h_{n_k}\circ g_{\al_0} \circ h_{n_k+1})(x))\\
&=\lim(\ti{h}_1\circ \ti{g}_{\al_0}\circ \ti{h}_2 \circ \ldots \circ \ti{h}_{n_k}\circ \ti{g}_{\al_0} \circ \ti{h}_{n_k+1})(\rho(x))\\
&=\lim\ti{f}_{n_k}(\rho(x))
\end{split}
\end{equation}

Therefore $\rho(z)\in \w(\rho(x))$ and hence $\rho(\w(x))\subset\w(\rho(x))$.

Since $\rho^{-1}$ is a conjugacy from $Y$ to $X$, by previous arguments we have the reverse inclusion. Thus $\w(\rho(x))=\rho(\w(x))$ for every $x \in X$.
\end{proof}

\begin{definition}
A subspace $Y \subset X$ is said to be (forward) {\it{invariant}} under $G$ if $g(y)\in Y$ for all $g\in G$ and $y\in Y$.
\end{definition}

If $G=\ <g_{\al}>_{\al}$ then for $Y$ to be invariant it is sufficient that $g(y)\in Y$ for all $g\in \{g_{\al}\}$ and $y\in Y$.

\begin{proposition}\label{invariant}
$\w(x)$ is invariant.
\end{proposition}

\begin{proof}
If $x \in X$ and $z\in \w(x)$ then there exists an  unbounded sequence $(f_{n_k})\subset G$ such that $f_{n_k}=h_1\circ g_{\al_0}\circ h_2\circ \ldots \circ h_{n_k}\circ g_{\al_0} \circ h_{n_k+1}$, for some $g_{\al_0}$ and $\displaystyle{\lim_{k\to \ity}}f_{n_k}(x)=z $.

For $g\in \{g_{\al}\}_{\al}$, since $g$ is continuous,
\begin{equation}\begin{split}\notag
g(z) &= g(\lim f_{n_k}(x)) \\
&=\lim ((g\circ h_1)\circ g_{\al_0}\circ h_2\circ \ldots \circ h_{n_k}\circ g_{\al_0} \circ h_{n_k+1})(x)\\
&=
\begin{cases}
\lim f_{n_k}\rq{}(x), & g \neq g_{\al_0}\\
\lim f_{n_k+1}\rq{}(x), & g=g_{\al_0}.
\end{cases}
\end{split}
\end{equation}

Since the sequences ($f_{n_k}\rq{}$) and ($ f_{n_k+1}\rq{}$) are also unbounded, we have $g(z)\in \w(x)$. Hence $\w(x)$ is invariant.
\end{proof}

\begin{proposition}
If  $O_G(x)$ is contained in a compact subset of X then $\w(x)$ is nonempty.
\end{proposition}

\begin{proof}
Let $(f_{n_k})$ be an unbounded sequence in the semigroup $G$. Then the sequence $(f_{n_k}(x))$ lies in a compact set containing $O_G(x)$. Hence there is a convergent subsequence whose limit is an $\w-$limit point of $x$. Thus $\w(x)$ is not empty.
\end{proof}

In the classical case of a continuous (or discrete) dynamical system, we see that $\w(x)$ is closed (provided $X$ is first countable) \cite{Alongi}. In the general case of continuous semigroup the closedness of $\w(x)$ is not clear. We shall  address this problem partially in this paper, where we establish the closedness of $\w(x)$ in case of a finitely generated semigroup.

\begin{proposition}
Let $G$ be a finitely generated semigroup and $x \in X$  then  $\w(x)$ is closed.
\end{proposition}

\begin{proof}
If $y \in \overline{\w(x)}$ and $G=\ <g_i>_{i=1}^m$, there exists a sequence $(z_n^i)_n\subset \w(x)$ such that $z_n^i \to y$, as $n \to \ity$. Where each $z_n^i=\displaystyle{\lim_{k\to \ity}}f_{n_k}^i (x)$ and $(f_{n_k}^i)\subset G$ is an unbounded sequence corresponding to fix $g_i$ for some $1 \le i \le m$.

Since $G$ is finitely generated, for some $1 \le i_0 \le m$, there is a subsequence $(z_{n_j}^{i_0})\subset (z_n^i)$ and $z_{n_j}^{i_0} \to y$, where each $z_{n_j}^{i_0}=\displaystyle{\lim_{k\to \ity}}f_{n_{j_k}}^{i_0} (x)$. Consider the diagonal subsequence $f_{m_k}:=f_{n_{k_k}}^{i_0}$. Then the sequence $(f_{m_k})$ is unbounded and $\displaystyle{\lim_{k\to \ity}f_{m_k}(x)=y}$. Thus $y \in \w(x)$.
\end{proof}

\begin{proposition}
Let $G$ be a finitely generated semigroup and $M$ is a nonempty compact subset of $X$. $M$ is minimal with respect to $G$ iff $M=\w(x)$ for every $x\in M$.
\end{proposition}

\begin{proof}
Let $x \in M$. Since $G$ is finitely generated and $M$ is compact and invariant being minimal, $\w(x)$ is nonempty closed and invariant subset of $M$. Since $M$ is minimal, $M=\w(x)$.

Conversely, let $M=\w(x)$ for every $x\in M$. If $N$ is a nonempty closed invariant subset of $M$ and $y\in N$, then $M=\w(y) \subset N$. Thus $M$ is a minimal subset of $(X,G)$.
\end{proof}

\begin{definition}
Let $(X,G)$  be a topological dynamical system. A point $x\in X$ is called a {\it{recurrent point}} with respect to $G$ if  $x\in \w(x)$.
\end{definition}

The set of all recurrent points of $(X,G)$ is denoted by $Rec(G)$.

\begin{theorem}
If $G$ is abelian then $Rec(G)$ is invariant.
\end{theorem}

\begin{proof}
Let $G$ be an abelian semigroup and $x\in Rec(G)$. Since $x$ is a recurrent point, $x\in \w(x)$. There exists an unbounded sequence $(f_{n_k})\subset G$ such that $x=\displaystyle{\lim_{k \to \ity}}f_{n_k}(x)$. Let $g$ be any generator of $G$. Since $g$ is continuous and $G$ is abelian, we have,
\begin{equation}\begin{split}\notag
g(x) &= g(\lim f_{n_k}(x)) \\
&=\lim g\circ f_{n_k}(x)\\
&=\lim f_{n_k}(g(x)).
\end{split}
\end{equation}

Thus $g(x)\in Rec(G)$. Hence $Rec(G)$ is invariant.
\end{proof}

\begin{theorem}
Let $(X,G)$ and $(Y,\ti{G})$ be two dynamical systems. If $\rho : X \to Y$ is a topological conjugacy then $\rho(Rec(G))=Rec(\ti{G})$.
\end{theorem}

\begin{proof}
Let $x\in Rec(G)$, that is, $x \in \w(x)$. By proposition \ref{conjugacy}, $\rho(x)\in \rho(\w(x))=\w(\rho(x))$. Thus $\rho(x) \in Rec(\ti{G})$.

Conversely, let $y \in Rec(\ti{G})$. Since $\rho^{-1}$ is a conjugacy from $Y$ to $X$, we have, $y \in \rho(Rec(G))$. Thus, $\rho(Rec(G))=Rec(\ti{G})$.
\end{proof}

\begin{definition}
A point $x\in X$ is called an {\it{escaping point}} if for given any compact set $K \subset X$ every unbounded sequence $(f_{n_k})\subset G$  leaves $K$ at $x$, that is, there exists $N(f_{n_k})\in \N$ such that $f_{n_j}(x) \notin K$ for each $n_j > N(f_{n_k})$.
\end{definition}

The set of all escaping points of $(X,G)$ is denoted by $Esc(G)$. Also as in the proof of proposition \ref{invariant}, with similar arguments, one can observe that $Esc(G)$ is invariant.

\begin{theorem}
$Esc(G)$ is invariant.
\end{theorem}

\begin{proof}
Let $x \in Esc(G)$ and $g \in \{g_{\al}\}_{\al}$. Since for given any compact set $K \subset X$ every unbounded sequence $(f_{n_k})\subset G$  leaves $K$ at $x$. Let $(f_{n_k})\subset G$  be an unbounded sequence such that $f_{n_k}=h_1\circ g_{\al_0}\circ h_2\circ \ldots \circ h_{n_k}\circ g_{\al_0} \circ h_{n_k+1}$, for fix $g_{\al_0}$. Then
\begin{equation}\notag
f_{n_k}\circ g=
\begin{cases}
f_{n_k}\rq{}, & g \neq g_{\al_0}\\
f_{n_k+1}\rq{}, & g=g_{\al_0}.
\end{cases}
\end{equation}

Since the sequences ($f_{n_k}\rq{}$) and ($ f_{n_k+1}\rq{}$) are also unbounded and $x$ is an escaping point, the sequence $(f_{n_k}\circ g)$ leaves $K$ at $x$. Hence  the sequence $(f_{n_k})$ leaves $K$ at $g(x)$. Thus $g(x) \in Esc(G)$.
\end{proof}

\begin{theorem}
Let $(X,G)$ and $(Y,\ti{G})$ be two dynamical systems. If $\rho : X \to Y$ is a topological conjugacy then $\rho(Esc(G))=Esc(\ti{G})$.
\end{theorem}

\begin{proof}
If $x \in Esc(G)$ and $\rho(K)$ is any compact subset of $Y$, for any compact $K \subset X$. Let $\ti{f}_{n_k}\subset \ti{G}$ be an unbounded sequence given by $\ti{f}_{n_k}=\ti{h}_1\circ \ti{g}_{\al_0}\circ \ti{h}_2 \circ \ldots \circ \ti{h}_{n_k}\circ \ti{g}_{\al_0} \circ \ti{h}_{n_k+1}$, for some $\ti{g}_{\al_0}$. Since $x \in Esc(G)$ and $\rho : X \to Y$ is a topological conjugacy, we have, $\ti{f}_{n_k}(\rho(x))=\ti{h}_1\circ \ti{g}_{\al_0}\circ \ti{h}_2 \circ \ldots \circ \ti{h}_{n_k}\circ \ti{g}_{\al_0} \circ \ti{h}_{n_k+1}(\rho(x))=\rho\circ h_1\circ g_{\al_0}\circ h_2\circ \ldots \circ h_{n_k}\circ g_{\al_0} \circ h_{n_k+1}(x)$ leaves $\rho(K)$. Therefore, $\rho(Esc(G))\subset Esc(\ti{G})$.

Since $\rho^{-1}$ is a conjugacy from $Y$ to $X$, by previous arguments we have the reverse inclusion. Thus $\rho(Esc(G))=Esc(\ti{G})$.
\end{proof}

{\bf Acknowledgments.}
I am thankful to my thesis adviser Sanjay Kumar for fruitful discussions. I am also thankful to the Department of Mathematics, Deen Dayal Upadhyaya College, University of Delhi for providing the research facilities.

\end{document}